\RequirePackage[english]{babel}
\documentclass[a4paper,leqno,twoside]{amsart}

\usepackage{amscd}
\usepackage{url}
\usepackage{tabularx}

\theoremstyle{plain}
\newtheorem{thm}{Theorem}
\newtheorem{cor}[thm]{Corollary}
\newtheorem{prop}[thm]{Proposition}
\newtheorem{lem}[thm]{Lemma}
\newtheorem*{conj}{Conjecture}
\newtheorem{pb}{Problem}

\theoremstyle{definition}

\newtheorem{ex}{Example}

\theoremstyle{remark}
\newtheorem{rmk}{Remark}
\newtheorem*{acks}{Acknowledgements}

\newcommand{\ie}{\textit{i.e. }}

\newcommand{\p}{\mathbb{P}}

\newcommand{\TT}{\mathbb{T}}

\newcommand{\Gr}{\mathbb{G}}

\newcommand{\OO}{\mathcal{O}}

\newcommand{\II}{\mathcal{I}}

\newcommand{\abs}[1]{\lvert#1\rvert}

\newcommand{\gen}[1]{\langle#1\rangle}

\DeclareMathOperator{\rk}{rk}

\DeclareMathOperator{\Sing}{Sing}

\def\cocoa{{\hbox{\rm C\kern-.13em o\kern-.07em C\kern-.13em o\kern-.15em A}}}
\DeclareMathOperator{\PGL}{\mathbb{P}GL}

\begin{document}

\title[Quadratic normality]{Congruences of lines  in $\p^5$,
quadratic normality,
and completely exceptional Monge-Amp\`ere equations}
\author{Pietro De Poi \and Emilia Mezzetti}

\thanks{This research was partially supported by  MiUR,
project  ``Geometria delle variet\`a algebriche e dei loro spazi di moduli''
for both authors, by Regione Friuli Venezia Giulia,
``Progetto D4'' for the first author, and by Indam project
``Birational geometry of projective varieties'' for the second one}

\address{Dipartimento di Matematica e Informatica\\
Universit\`a degli Stud\^\i\ di Trieste\\
Via Valerio, 12/1\\
I-34127 Trieste\\
Italy}
\email{depoi@dmi.units.it} \email{mezzette@units.it}


\begin{abstract}
The existence is proved of two new families of locally
Cohen-Macaulay sextic threefolds
in $\p^5$, which are not quadratically normal. These threefolds arise naturally
in the realm of first order congruences of lines as focal loci
and in the study of the
completely exceptional Monge-Amp\`ere equations.
One of these families comes from a smooth
congruence of multidegree $(1,3,3)$ which is a smooth
Fano fourfold of index two and
genus $9$.
\end{abstract}

\maketitle
\section*{Introduction}
In 1901 Francesco Severi \cite{Sev} proved his celebrated theorem saying that
the unique surfaces in $\p^5$ whose secant variety is strictly contained
in $\p^5$ are the Veronese surface and the cones.
It follows that if $X$ is a nondegenerate surface in $\p^4$,
then it is  \emph{linearly normal} (\ie $h^1(\mathcal{I}_X(1))= 0$)
unless  it is the (projected) Veronese surface.
Note that a projection of a non-degenerate cone
fails to be isomorphic in the vertex, so cones in $\p^4$ are
also linearly normal.
As an easy  corollary, it follows that
if $X$ is a nondegenerate threefold in $\p^5$, then it is linearly
normal.

An interesting feature of the projected Veronese surface $V$ is that
its trisecant lines form a \emph{linear congruence of lines} in $\p^4$,
\ie a $3$-dimensional linear section of $\Gr(1,4)$, the Grassmannian
of lines in $\p^4$. This allows to characterize $V$ as the degeneracy locus
of a general bundle morphism
$\phi\colon {\OO}_{ {\p}^4}^{\oplus 3}\to{\Omega}_{{\p}^4}(2)$  (\cite{O}).
It gives also the following $\Omega$-resolution for the ideal of $V$:
\begin{equation*}
0 \to  {\OO_{ {\p}^4}(-1)}^{\oplus 3}\to
{\Omega}_{{\p}^4}(1)\to\mathcal{I}_V(2)\to 0,
\end{equation*}
which immediately implies that none of the quadrics containing
the hyperplane section of $V$  lifts to a quadric containing $V$ (\cite{Cha}).

In $\p^5$ an analogous construction brings to
the definition of the \emph{Palatini scroll},
a smooth threefold $X$ defined as degeneracy locus of a
general morphism
$ {\OO}_{ {\p}^5}^{\oplus 4}\to{\Omega}_{{\p}^5}(2).$
$X$ results to be  \emph{not quadratically normal}, \ie
$h^1(\mathcal{I}_X(2))\neq 0$.

These analogies motivated the following question,
posed by Peskine and Van de Ven,
on the non necessarily smooth
threefolds in $\p^5$:
\begin{pb}\label{pb1}
Is the Palatini scroll the unique threefold in $\p^5$ which is not
quadratically
normal? If it is not the unique example,  classify such
threefolds
(see \cite{Sch}, where the question is ascribed to Peskine).
\end{pb}

More generally, one can ask to classify the codimension two subvarieties
$X$ in $\p^n$,  which are not $(n-3)$-normal (\ie
$h^1(\mathcal{I}_X(n-3))\neq 0$).

Subsequently, F. L. Zak in \cite{Z}
has generalized Problem~\ref{pb1}, posing
the following conjecture
about $j$-normality of non-degenerate projective
$m$-dimensional subvarieties  in $\p^n$.

\begin{conj}\label{cj:1} Let $i$ and $j$ be two integers such that
$i\ge 1$ and $j\ge 0$; then
\begin{enumerate}
\item $H^i(\p^n,\II_X(j))=0$ for $i+j<\frac{m}{n-m-1}$;
\item\label{eq:cj1} for $i+j=\frac{m}{n-m-1}$ it is possible to ``classify''
all the varieties for which $H^i(\p^n,\II_X(j))\neq 0$.
\end{enumerate}
\end{conj}

Zak also suggested that examples on the boundary of the conjecture
could be constructed as focal loci of some \emph{first order congruences}
in $\p^n$ (a congruence is a flat family of lines of dimension $n-1$ and its order
is the number of its lines passing
through a general point of the space; for generalities on congruences, see \cite{DP6}).
As  we have seen in the case of $\p^4$, all the known examples are
varieties whose $l$-secant lines (with $l$
appropriate integer) give linear congruences, which are special examples
of congruences of order one.  In the Chow ring of $\Gr(1,n)$,
every congruence  is a linear combination with integer coefficients
of the Schubert cycles of dimension $n-1$.
The sequence of the coefficients $(a_0, a_1,\dotsc, a_k)$,
$k=[\frac{n-1}{2}],$  is called the multidegree of the congruence,
the first integer $a_0$ is precisely its order.
For instance,  the multidegree of a linear congruence
$B$ in $\p^5$ is $(1,3,2)$,
this means precisely that the  lines of $B$ contained
in a general hyperplane fill a hypersurface of degree $3$,
and that the number of lines contained in a general $\p^3$ is $2$.
Since  $B$ is formed precisely by the $4$-secant lines of its focal
locus $X$, a Palatini threefold in $\p^5$, we find in this way that
$X$ cannot be contained in a cubic hypersurface while,
conversely, its hyperplane section is.

This remark shows the connections with the \emph{lifting problem}:
given an integral $\nu$-dimensional variety $X\subset \p^n$,
it concerns finding conditions on $d,\nu, n$ and $\sigma$ so
that any degree $\sigma$ hypersurface in $\p^{n-1}$
containing the hyperplane section $S=X\cap H$ of $X$
($H$ a general hyperplane) lifts to a hypersurface
of degree $\sigma$ containing $X$.
In view of the  cohomology of the exact sequence of sheaves
\begin{equation*}
0\to\II_X(k)\xrightarrow{\cdot H} \II_X(k+1)\to\II_S(k+1)\to 0,
\end{equation*}
it is clear that, if $X$ is $k$-normal, then all hypersurfaces
of degree $k+1$ containing
$S$ lift to hypersurfaces of the same degree containing $X$ (but not
vice-versa). Instead, if $k+1$ is a \emph{non-lifting level}
for $X$ (\ie the restriction map $H^0(\mathcal{I}_X(k+1))\rightarrow
H^0(\mathcal{I}_{S}(k+1))$ is not onto for a general
hyperplane $H$), then $X$ is necessarily non $k$-normal.

The main results about the lifting problem have been obtained in the case
of  subvarieties of codimension two. Laudal \cite{lau},
Gruson-Peskine \cite{gp},
Strano \cite{str} solved the problem in the case of curves in $\p^3$,
whereas the case of higher dimensional varieties was studied,
among others, by Mezzetti-Raspanti \cite{MR},
Mezzetti \cite{Mez2}, \cite{Mez},  Chiantini-Ciliberto \cite{CC},
Roggero \cite{Rug1}, \cite{Rug}, \cite{rug}.

The  result  for threefolds in $\p^5$  in the case $\sigma=3$ is the following
(\cite{Mez},
\cite{rug}):

\begin{thm}
If $X$ is a \emph{locally Cohen-Macaulay}
(\emph{lCM} for short in the following)
integral threefold in $\p^5$, if $\deg(X)>7$ and
$h^0(\mathcal{I}_{X\cap H}(3))\neq 0$, then
$h^0(\mathcal{I}_{X}(3))\neq 0$. Moreover the only
lCM threefolds of degree $7$ with
$\sigma=3$  for which
  the assertion is not true are the
Palatini threefold and its degenerations.
\end{thm}

Our purpose here is to complete this result finding
all the lCM threefolds  with $h^0(\mathcal{I}_{X\cap H}(3))\neq 0$
and $h^0(\mathcal{I}_{X}(3))= 0$.

We prove:

\begin{thm}\label{teo1}
Let $X$ be a non-degenerate lCM integral threefold of $\p^5$ of
degree $d\leq 6$ and (arithmetic) sectional genus $\pi$. Then:
\begin{enumerate}
\item if either $d\leq 5$ or $d=6$ and $\pi\neq 1,2$,
$X$ is contained in a quadric or in a cubic and is $2$-normal;
\item if $d=6$ and $\pi=2$, $X$ is contained in a quadric or in a cubic and
$h^1(\II_X(2))\leq 1$.
\end{enumerate}
\end{thm}

\begin{thm}\label{teo2}
There exist two families of lCM threefolds, non singular in codimension one
of degree $6$ in $\p^5$
with sectional genus one, such that $h^1(\II_X(2))=1$ and
$h^0(\mathcal{I}_{X\cap H}(3))\neq 0$, but
$h^0(\mathcal{I}_{X}(3))= 0$.
\end{thm}

Our examples give an answer in the negative to Problem \ref{pb1}, at
least if one allows singularities.
In fact, both examples are singular and non singular in codimension
$1$. A threefold of the first family is obtained from a particular
case of a linear congruence of lines, the focal locus is reducible
in two components, one of them is a \emph{parasitical} $\p^3$ (see
Section~\ref{sec:2}). The second family comes from the congruences
which are associated to the completely exceptional Monge-Amp\`ere
systems of differential equations. The existence of this family
was suggested by Agafonov-Ferapontov:  in their article \cite{AF2}
they introduce a surprising construction which allows to associate
to a system of PDE of conservation laws in $n-1$ variables a
congruence of lines $B$ in $\p^n$. If the system is of Temple
class (also called \emph{T-system}, see Section~\ref{sec:3} for
the definition), then $B$ has order one, and the focal locus
is a subvariety $X$ in $\p^n$ of codimension two, such that the
lines of $B$ through a general point of $X$ form a planar
pencil of lines (see \cite{dm}). If $n\leq 4$, it is proved in
\cite{AF2} that a congruence associated to a T-system is always
linear. The example we exhibit shows that this is no longer true
in $\p^5$, indeed we get a congruence of multidegree $(1,3,3)$. It
results to be an irreducible component of a (special) quadratic
congruence, \ie of a subvariety of $\Gr(1,5)$, which is cut by
three hyperplanes and one quadric. This congruence, seen as a
subvariety of the Grassmannian, is also interesting from other
points of view. It is a new example of a smooth variety of
dimension $4$ covered by lines, such that the number of lines
passing through its general point is $4$.
It results to be a Fano fourfold of index two and genus $9$.
This was studied by Mukai in \cite{Mu}, who had given an embedding
of it in the Grassmannian of $2$-planes in $\p^5$.
We note that Mukai's embedding is obtained by showing the existence
of a ``good'' vector bundle of rank three on the Fano variety,
while our embedding comes from a rank two vector bundle.

We also show that our examples exhaust all the possibilities
for integral threefolds coming from congruences of type $(1,3, a_2)$,
although we do not exclude the possibility of having more than two
families in Theorem \ref{teo2}.

This article is structured as follows: in Section~\ref{sec:1},
we prove Theorem~\ref{teo1}
about threefolds of degree $\leq 6$ in $\p^5$.
In Section~\ref{sec:2} we construct the congruences giving rise to
the two families of non-$2$-normal threefolds of degree $6$
and sectional genus $1$.
Finally in Section~\ref{sec:3} we explain the connections with the
completely exceptional Monge-Amp\`ere equations.

\begin{acks}
We wish to thank Joseph~Landsberg and
Dario~Portelli   for interesting discussions.
\end{acks}

\section{Threefolds of degree $\leq 6$ in $\p^5$}\label{sec:1}

In this section $X$ will always denote a locally Cohen-Macaulay
(lCM for short) threefold of $\p^5$.

If $\deg X\le 4$,  then $X$ is contained in a
quadric (\cite{Rug}). This implies that it is arithmetically Cohen-Macaulay
(see \cite{KW}, Theorem~2.1).
It remains to consider the cases $d=5,6$.

The following simple lemma will allow us to prove the $2$-normality
of the projections of some rational normal scrolls.

\begin{lem}\label{lem:knorm}
Let $X\subset\p^n$ be a (nondegenerate) projective variety
and  $Y\subset X$ a codimension one subvariety.
If $Y$ is $k$-normal and one of the following assumptions is satisfied:
\begin{enumerate}
\item\label{hyp} $Y$ is a
hyperplane section of $X$ and $X$ is $(k-1)$-normal,
\item\label{mag} $Y$ is such that $\deg(Y)>\deg (X)$ and there does not
exist any hypersurface $V$ of degree $\le k$ which contains $Y$ but not $X$,
\end{enumerate}
then $X$ is $k$-normal too.
\end{lem}

\begin{proof}
Case~\eqref{hyp} follows immediately from the cohomology sequence of
\begin{equation*}
0\to\II_X(k-1)\xrightarrow{\cdot H} \II_X(k)\to\II_Y(k)\to 0.
\end{equation*}

Let us now assume the hypothesis of case~\eqref{mag}.
Consider the $k$-tuple embedding $v_k(X)$ of $X$ in $\p^N$,
$N:=\binom{n+k}{k}-1$. Its linear span $\gen{v_k(X)}$ is a $\p^M$, $M\le N$.
If $X\subset\p^n$ is not $k$-normal, then $v_k(X)$ is not linearly normal, \ie
there exist a linear projection
$\pi\colon \p^{M+1}\dashrightarrow \p^M$ and a nondegenerate variety
$\tilde{X}\subset\p^{M+1}$ such that $\pi\mid_{\tilde{X}}$
is an isomorphism onto $v_k(X)$.
Therefore, for each subvariety $Y$ of $X$,
there exists a subvariety $\tilde{Y}\subset\tilde{X}$ such
that  $\pi\mid_{\tilde{Y}}$
is an isomorphism onto $v_k(Y)$.
In particular, if $\dim(\gen{\tilde{Y}})>\dim(\gen{v_k(Y)})$, $Y$ is not
$k$-normal.

We observe that, in general, if $Z\subset \p^n$ is a nondegenerate variety,
then  $v_k(Z)\subset \p^N$ is degenerate  (\ie $M<N$) if and only if
there exists a hypersurface of degree $k$, $V\subset\p^n$, such that
$Z\subset V$.

Assumption~\eqref{mag} implies that
$\gen{v_k(X)}=\gen{v_k(Y)}$. If now we suppose that $X$ is not $k$-normal,
there exists a $\tilde{Y}$ as above, and
$\deg(\tilde{Y})=\deg(v_k(Y))>\deg(v_k(X))=\deg(\tilde{X})$. Then,
either $\gen{\tilde{X}}=\gen{\tilde{Y}}$ and
$\gen{v_k(X)}=\gen{v_k(Y)}$ and $Y$ would not be $k$-normal, or
$\gen{\tilde{Y}}$ is a hyperplane in $\p^{M+1}$; but this cannot
occur, by the assumption on the degree, since $\tilde{Y}\subset
\gen{\tilde{Y}}\cap \tilde{X}$.
\end{proof}

In the next two theorems we collect the connections between lifting problem and $k$-normality and the known results about the lifting
problem for varieties of dimension $\leq 3$. We will use the
following notation: if $X$ is a projective variety, we put
\begin{equation*}
s(X):=\min\{k\mid h^0(\II_X(k))\neq 0\}
\end{equation*}
and
\begin{equation*}
\sigma(X):=\min\{k\mid h^0(\II_{X\cap H}(k))\neq 0\}
\end{equation*} where $H$ is a general hyperplane.  We will say that an integer number
$h$ is a \emph{non-lifting level}
 for $X$ if the restriction map $H^0(\mathcal{I}_X(h))\rightarrow
H^0(\mathcal{I}_{X\cap H}(h))$ is not surjective, when $H$ is a general
hyperplane.

\begin{thm}\label{thm:connections}
Let $X\subset \p^n$ be an integral variety  and $S$ be a general hyperplane section of $X$. Let $k\geq 1$ be an integer number.
\begin{enumerate}
\item If $X$ is $k$-normal, then all hypersurfaces of degree $k+1$ in $\p^{n-1}$ containing $S$ lift to a hypersurface of degree $k+1$ in $\p^n$ containing $X$;
\item if $k+1$ is a non-lifting level for $X$, then $X$ is not $k$-normal.
\end{enumerate}
\end{thm}

\begin{thm}\label{thm:lifting}
Let $X\subset\p^{n+2}$ ($n\ge 1$) be a nondegenerate lCM
projective variety of codimension $2$ and degree $d$.

If $\sigma:=\sigma(X)<s(X)$,
then
\begin{enumerate}
\item if $\dim(X)=1$,  $d\le \sigma^2+1$;
\item if $\dim(X)=2$,  $d\le \sigma^2-\sigma+2$;
\item if $\dim(X)=3$,  $d\le \sigma^2-2\sigma+4$.
\end{enumerate}
\end{thm}

The proof of Theorem \ref{thm:connections} follows immediately from the cohomology of the exact sequence
\begin{equation*} 0\to\II_X(k)\xrightarrow{\cdot H} \II_X(k+1)\to\II_S(k+1)\to 0.
\end{equation*}
For the proof of Theorem \ref{thm:lifting},
we refer to the papers cited in the
introduction.
\begin{rmk}
Note that the second assertion of Theorem \ref{thm:connections} cannot be reversed, because there exist
varieties $X$ such that $k+1$ is a lifting level (for a certain $k$), but nevertheless $X$ is not $k$-normal.
For instance, let $X$ be a smooth rational curve in $\p^3$ of degree $d\ge 5$ and assume that it is contained in a quadric. Then $X$ is not linearly normal but clearly $2$ is a lifting level for $X$.
\end{rmk}

We now start to prove Theorem  \ref{teo1} of the introduction. Next proposition takes care of the threefolds with sectional genus zero.

\begin{prop}\label{prop:rns}
Let $X'\subset\p^n$ be a rational normal scroll of dimension $3$
(therefore of degree $n-2$),
with $n\le 8$;
then a lCM linear projection $X$ of $X'$ to $\p^5$  is $2$-normal.
Moreover, $h^0(\II_X(3))>0$.
\end{prop}

\begin{proof}
By Lemma~\ref{lem:knorm}, applied to a general hyperplane section,
it is sufficient to show the $2$-normality for a projection of a
$2$-dimensional rational normal scroll $Y'$
to a $\p^4$. Let us call this projected surface $Y\subset\p^4$. We have
$Y'=S(b_0,b_1)$, $0\le b_0\le b_1$, $b_0+b_1=n-2$. Then
the section $C_0$ (see \cite{H}, page 373) and the fiber $f$ generate the Picard
group of the ruled surface $Y'$. In particular, the hyperplane and canonical
divisors are
(numerically) $C_H=C_0+b_1f$ and $K=-2C_0+(-2-e)f$, where $-C_0^2=e=b_1-b_0$.
Moreover $K^2=8$ (see again \cite{H}, pages 373--374).

In order to apply again  Lemma~\ref{lem:knorm}, we will show that
there exists a curve $C'\subset Y' \subset \p^{n-1}$
such that $\deg C'>\deg  Y'$ and
$C=\pi(C')$ is $2$-normal and not contained in a
quadric.
Let us take $C'\in\abs{C_0+(b_1+9-n)f}$; it is immediate to see, for example
from the adjunction formula, that $C'$ has arithmetic genus zero. Moreover,
$\abs{C'}$ is very ample (see \cite{H}, Theorem~V.2.17(c)).
So, if $C'$ is general in its linear system, it is a smooth rational curve of
degree $C'\cdot (C_0+b_1f)=b_0+b_1+9-n=7$.


By Riemann-Roch $h^0(\OO_C(2))=15=h^0(\OO_{\p^4}(2))$,
so $C$ is not $2$-normal iff it is contained in a quadric.
Let us consider then the hyperplane section $C_H$ of $Y$; if $C_H$
is contained in a
quadric,
 we can apply
Theorem~\ref{thm:lifting}
 to get that $Y$ and $X$ are contained in a quadric too. Therefore, $X$ is
arithmetically Cohen-Macaulay (\cite{KW}, Theorem~2.1).
So, let us suppose $h^0(\II_{C_H}(2))=0$.

If $n=8$, $C$ is linearly equivalent to $C_H\cup f$.
Then, 
\begin{equation*}
\frac{\II(C_H)}{\II(C_H)\cap\II(f)}\cong
\frac{\II(C_H)+\II(f)}{\II(f)}
\cong\frac{\II(C_H\cap f)}{\II(f)}
\end{equation*}
and so the following sequence
\begin{equation*}
0\to \II_{C_H\cup f}\to \II_{C_H}\to\II_{P\mid f}\to 0,
\end{equation*}
where $P=C_H\cap f$, is exact. We have that $\II_{P\mid f}\cong \OO_f(-1)$, and
since $C_H$ is degenerate and it is not contained in a quadric, we obtain
\begin{equation}\label{cu}
0\to H^0(\II_{C_H\cup f}(2))\to H^0(\OO_{\p^4}(1)) \to H^0(\OO_f(1))\to
H^1(\II_{C_H\cup f}(2))\to 0.
\end{equation}
Now, $h^0(\OO_{\p^4}(1))=5$, $h^0(\OO_f(1))=2$ and
$h^1(\II_{C_H\cup f}(2))\ge h^1(\II_{C}(2))$ by semicontinuity (Theorem~III.12.8
of \cite{H}). If, by contradiction, $h^1(\II_{C}(2))\neq 0$, then
$h^1(\II_{C_H\cup f}(2))\neq 0$, so  $h^0(\II_{C_H\cup f}(2))>3$, hence there exists
at least one irreducible quadric in $\p^4$
containing $C_H\cup f$,
in particular $C_H$, a contradiction. Therefore, $C$ is $2$-normal and not
contained in a quadric, and by Lemma~\ref{lem:knorm} $X$ is $2$-normal.

If $n=7$, $C$ is linearly equivalent to $C_H\cup f_1\cup f_2$, where $f_1$ and
$f_2$ are two fibers. Arguing as in the previous case, we obtain an
exact sequence
\begin{equation*}
0\to H^0(\II_{C_H\cup f_1\cup f_2}(2))\to H^0(\OO_{\p^4}(1)) \to
\begin{matrix}
H^0(\OO_{f_1}(1))\\
\oplus \\
H^0(\OO_{f_2}(1))
\end{matrix}
\to
H^1(\II_{C_H\cup f_1\cup f_2}(2))\to 0.
\end{equation*}
If $h^1(\II_{C}(2))\neq 0$, this time we get  $h^0(\II_{C_H\cup f_1\cup f_2}(2))>1$
and, since $\gen{f_1,f_2}\cong\p^3$, we get again an irreducible quadric
which contains $C_H$.

If $n=6$, $C$ is linearly equivalent to $C_H\cup f_1\cup f_2\cup f_3$;
in this case, it suffices to observe that, by semicontinuity,
$h^0(\II_{C_H\cup f_1\cup f_2\cup f_3}(2))\ge h^0(\II_C(2))$, so, if $C$ is contained
in a quadric, also $C_H\cup  f_1\cup f_2\cup f_3$ is contained in a quadric.
This quadric is irreducible since the three fibers generate $\p^4$. So we get
again a contradiction.

The last assertion follows from the fact that $h^0(\II_{C_H}(3))>0$ by Riemann
Roch, and by the fact that $X$ and $Y$ are quadratically normal.

\end{proof}

\subsection{Threefolds of degree $5$ in $\p^5$}

We shall now prove that a lCM threefold $X$ of degree $d=5$
 is always  contained  either in a quadric or in a cubic and is $2$-normal.

\begin{thm}\label{thm:51}
Let $X$ be a lCM threefold of degree $5$ and sectional genus $\pi$ in $\p^5$ not
contained in a quadric;
then $X$ is contained in a cubic and is $2$-normal. Moreover, if $\pi=1$ then
$h^0(\II_X(3))=5$.
\end{thm}

\begin{proof}
If the geometric genus of the curve section $C$ is zero, then we apply
Proposition~\ref{prop:rns}.

If $\pi=2$, then $C$ is a curve of maximal genus, and therefore it is
arithmetically Cohen-Macaulay (aCM for short), and so also the surface
section $S$ and $X$ are aCM.

If $\pi=1$, let us consider the long exact sequence of cohomology 
\begin{equation}
0\to H^0(\II_X(3))\to H^0(\II_S(3))
\to H^1(\II_X(2))\to \dotsb.
\end{equation}
The intersection of $X$  with a general
plane is formed by $5$ points, which are contained in a conic. If this conic
  lifts to a quadric containing the curve section of $X$, by
Theorem~\ref{thm:lifting} for surfaces and threefolds
it lifts also to a quadric containing $X$. So it
does not lift, then the
curve section $C=H\cap H'\cap X$ has genus one, and it is bilinked to
a couple of  skew lines $\ell_1$ and $\ell_2$:
$C\underset{(3,3)}\sim C_4\underset{(2,3)}\sim\ell_1\cup\ell_2$,
where $C_4$ is a rational quartic.

Now, from
\begin{equation*}
0\to\II_S(1)\to \II_S(2)\to\II_C(2)\to 0,
\end{equation*}
since $ H^0(\II_C(2))=0$, we get
\begin{equation}
0\to H^1(\II_S(1))\to  H^1(\II_S(2))\to H^1(\II_C(2))\to \dotsb.
\end{equation}
By Severi's Theorem $H^1(\II_S(1))= 0$.
But also $H^1(\II_C(2))=0$:
indeed by
\begin{equation*}
0\to\II_C(2)\to \OO_{\p^3}(2)\to\OO_C(2)\to 0,
\end{equation*}
we get
\begin{equation}
0\to H^0(\II_C(2))\to  H^0(\OO_{\p^3}(2))\to H^0(\OO_C(2))\to
H^1(\II_C(2))\to 0
\end{equation}
and, using Riemann-Roch, we conclude that $ h^0(\II_C(2))=h^1(\II_C(2))=0$.

Then $H^1(\II_S(2))=0$ so $ h^0(\II_S(3))=h^0(\II_C(3))$.
Now, $h^0(\II_C(3))=5$ (by mapping cone), so also $ h^0(\II_S(3))=5$.

 From the exact sequence
\begin{equation}
0\to H^0(\II_X(1))\to  H^0(\II_X(2))\to H^0(\II_S(2))\to H^1(\II_X(1))
\to H^1(\II_X(2))\to 0
\end{equation}
and  $H^0(\II_X(1))=0$, we obtain $H^1(\II_X(2))\cong H^1(\II_X(1))=0$;
and analogously, by
\begin{equation}
0\to H^0(\II_X(2))\to  H^0(\II_X(3))\to H^0(\II_S(3))\to H^1(\II_X(2))
\to 0
\end{equation}
we conclude that
$h^0(\II_X(3))=5$.
\end{proof}

\subsection{Threefolds of degree $6$ in $\p^5$}

By Proposition~\ref{prop:rns} it  suffices to analyze the cases of
positive sectional genus $\pi$. By Castelnuovo bound, $\pi\le 4$.
If $\pi=4$, then the curve section is aCM, so also
the threefold $X$ is aCM.

If $\pi=3$, then with computations similar to those made in the proof of
Theorem~\ref{thm:51}
we conclude that also in this case $X$ is quadratically normal and, if it is
not contained in a quadric, it is contained in a cubic.

We will show in the next section that there do exist threefolds of degree $6$
and sectional genus one which are not quadratically normal
(necessarily singular); so we analyze now the case of sectional
genus two:
\begin{prop}
Let $X$ be a threefold of degree $6$ and sectional genus two. Then
$h^0(\II_X(3))>0$ and $h^1(\II_X(2))\le 1$.
\end{prop}

\begin{proof}
Let $S,C$ and $\Gamma$ be, respectively, the general linear sections of
dimensions $2,1$ and $0$.

If $h^0(\II_\Gamma(2))>0$, by Theorem~\ref{thm:lifting},
we have as usual that also $X$ is
contained in a quadric and is quadratically normal, and we are done.

Let us suppose now  $h^0(\II_\Gamma(2))=0$. Then $\II_\Gamma$ has the
following minimal free resolution:
\begin{equation*}
0\to 3\OO_{\p^2}(-4)\to 4\OO_{\p^2}(-3)\to\II_\Gamma\to 0,
\end{equation*}
with $h^0(\II_\Gamma(3))=4$. From $h^0(\II_C(2))=0$ and Riemann-Roch,
we have $h^1(\II_C(1))=h^1(\II_C(2))=1$. Moreover, from \cite{GLP},
$h^1(\II_C(k))=0$ $\forall k\ge 3$. In particular, $h^0(\II_C(3))=3$.

Note that  $h^1(\II_S(2))\le 1$, because we have
\begin{equation*}
0\to H^1(\II_S(1))\to H^1(\II_S(2))\to  H^1(\II_C(2))\to \dotsb
\end{equation*}
with $S$ linearly normal and $h^1(\II_C(2))=1$.

From the exact sequence
\begin{equation}\label{sq:3}
0\to H^0(\II_S(3))\to  H^0(\II_C(3))\to H^1(\II_S(2))\to H^1(\II_S(3))
\to 0
\end{equation}
we can deduce now that  $h^0(\II_S(3))\ge 2$.
Finally, let us consider the sequence
\begin{equation*}
0\to H^0(\II_X(3))\to  H^0(\II_S(3))\to H^1(\II_X(2))
\to \dotsb.
\end{equation*}
Since $h^1(\II_X(2))\le h^1(\II_S(2))\le 1$, we get
  $h^0(\II_X(3))>0$.
\end{proof}

We have concluded in this way the proof of Theorem~\ref{teo1}
of the introduction.

\begin{cor}\label{cor:one}
The congruence of the $4$-secant lines of a lCM threefold of
degree $5$ or of degree $6$ with $\pi\neq 1$ cannot have positive order.
\end{cor}

\begin{rmk}
We observe that we can slightly weaken the hypothesis of Theorem~\ref{teo1}
in the case $d=6$. In fact in Proposition~\ref{prop:rns} we have proved that
if the geometric sectional genus is zero, then $X$ is quadratically normal and
contained in a cubic.
\end{rmk}

\begin{rmk}
In \cite{Ko}, it has been shown that if a lCM threefold $X\subset\p^5$ is
contained in a cubic hypersurface which is also lCM, then $X$ is aCM.
Moreover, if $X$ is smooth, the assumption lCM on the cubic
hypersurface can be removed.
\end{rmk}

To state the following proposition, we need to recall two definitions about
first order congruences of lines $B$ in $\p^n$
from \cite{DP4}, and precisely  those of \emph{parasitical scheme} and of
\emph{pure focal locus}.
A parasitical scheme is a subscheme of the focal locus, of dimension at least
two, covered by lines of $B$
and that is not met by a general line of the congruence.
The \emph{pure focal locus} is
the union of the components of the focal locus which are not parasitical.
Note that
therefore the general lines of $B$ are lines meeting $n-1$ times the
pure fundamental locus (see \cite{DP4}, Proposition~2.1).

\begin{prop}
Let $B$ be a first order congruence of lines in $\p^5$ such that its pure focal
locus $X$ is an irreducible threefold of degree six.
Then, the sectional genus of $X$ is one and
the multidegree of $B$ is $(1,3,a)$, with $a\le 3$.
\end{prop}

\begin{proof}
$X$ must be of sectional genus one, by Corollary~\ref{cor:one}.
If $C\subset\p^3$ is a smooth connected sextic curve of genus one, then the
expected number of $4$-secant lines is $3$, see \cite{LB}.

Moreover we have $h^0(\II_C(3))=2$, since $C$ is cubically normal by
\cite{GLP}.
From sequence~\eqref{sq:3} we deduce $h^0(\II_S(3))=1$, since otherwise
$S$, and therefore $X$, would be quadratically normal, therefore, the
second multidegree is $3$.

The third multidegree  $a$ is necessarily $\le 3$,
since we have observed that the expected
value is $3$, but some of these $3$ lines could be parasitical.
\end{proof}

We will see in Section~\ref{sec:2} an example with a parasitical
line. We will also show the existence of congruences of multidegree $(1,3,a)$ for
any $0\le a\le 3$.

\section{Linear and quadratic congruences}\label{sec:2}

We recall that a \emph{congruence of lines} in $\p^5$ is a flat family of lines
of dimension $4$. A \emph{linear congruence} is a congruence obtained by
a linear section---with a linear space $\Delta$ of dimension $10$---of
the Grassmannian
$\Gr(1,5)\subset\p^{14}$. Linear congruences in $\p^5$ are studied and
classified in \cite{dm}.

We recall that such a classification is obtained considering
$\check\Gr(1,5)\subset\check\p^{14}$ which is a cubic hypersurface, the
\emph{Pfaffian}, intersected with $\check\Delta\cong\p^3$.
Then, we have classified the linear congruences considering the
surface $Y:=\check\Gr(1,5)\cap\check\Delta$ and its possible positions in
  $\check\Gr(1,5)$ and with $\Sing(\check\Gr(1,5))\cong\Gr(3,5)$.

Among all the cases considered, particularly interesting is the
case 5.2.1 of \cite{dm}:  in this situation $Y$ is reducible, with
$Y=\pi\cup Q$, where $\pi$ is a plane, $Q$ is a quadric, and
$Y\cap \Gr(3,5)=\emptyset$. Let
\begin{equation*}
\gamma\colon \check\Gr(1,5)\dashrightarrow \Gr(1,5)
\end{equation*}
be the Gauss map. Its
restriction to the plane $\pi$ is a degree $2$ Veronese morphism, hence its
image is a Veronese surface. It parametrizes
the secant lines of a skew cubic curve
$C$ embedded in a three dimensional linear space $L\subset\p^5$.

Let us recall that the linear spaces contained in
$\check\Gr(1,5)\setminus \Gr(3,5)$
can be interpreted as linear spaces of
$6\times 6$ skew-symmetric matrices of constant rank $4$.
They are classified in \cite{MM}.
 In particular, the
planes are distributed in $4$ $\PGL_6$-orbits,
all of dimension $26$, and the one we are
considering is  generated for instance by the plane
\begin{align}\label{eq:mat}
\pi_t&=\begin{pmatrix} 0 & a & b & c & 0 & 0\\
      & 0 & 0 & 0 & 0 & 0\\
&  & 0 & 0  & 0 & a\\
&   &  & 0 & 0 & b\\
&  &  &  & 0 & c\\
&  &  &  &  & 0
\end{pmatrix}\\
&=ah_1+bh_2+ch_3.
\end{align}
Dually, $\check\pi_t=H_1\cap H_2 \cap H_3$ is a $11$-dimensional linear space,
which is tangent to the Grassmannian along the Veronese surface
$\gamma(\pi_t)$.

In terms of coordinates, in the Pl\"ucker embedding,
by \eqref{eq:mat}, the ideal of $\check\pi_t$ is
\begin{equation}\label{eq:pit}
  I(\check\pi_t) = \gen{p_{01}+p_{25},p_{02}+p_{35},p_{03}+p_{45}}.
\end{equation}
Let us denote by
$\Gamma$ the intersection
\begin{equation}\label{eq:gamma}
\Gamma:=\check\pi_t\cap \Gr(1,5).
\end{equation}
It is a $5$-dimensional family of lines.

We note that in this
case the Veronese surface parametrizes the secant lines of the twisted cubic
$C$
contained in
the $3$-dimensional subspace $L$ of equations $x_0=x_5=0$ defined by
\begin{equation}\label{cubic}
\rk\begin{pmatrix}
x_1 & x_2 & x_3\\
x_2 & x_3 & x_4
\end{pmatrix}<2.
\end{equation}

The lines of $\Gamma$ can be described as
follows:

\begin{lem}\label{lem:gp}
Let us denote by $\Gamma_p$ the lines of $\Gamma$ through $P$; then
\begin{enumerate}
\item if $P\in \p^5\setminus L$, $\Gamma_P$ forms a pencil of lines in
a plane that intersects $C$ in only one point;
\item if $P\in L\setminus C$,  $\Gamma_P$ is the star of lines through $P$
  contained in $L$;
\item if $P\in C$, $\dim(\Gamma_P)=3$ and  $\Gamma_P$ generates a hyperplane of
$\p^5$ (which contains $L$).
\end{enumerate}
\end{lem}

\begin{proof}
If $P=(a_0:\dotsb :a_5)\in  \p^5$, then, by \eqref{eq:pit}, and
recalling that $p_{ij}=x_ia_j-x_ja_i$, the
lines through $P$ belonging to $\Gamma$ are contained in the linear space
defined by
\begin{align}
a_1x_0-a_0x_1+a_5x_2-a_2x_5&=0\\
a_2x_0-a_0x_2+a_5x_3-a_3x_5&=0\\
a_3x_0-a_0x_3+a_5x_4-a_4x_5&=0.
\end{align}
Now, the matrix
of the coefficients
\begin{equation}
A=\begin{pmatrix}
a_1 & -a_0 & a_5 & 0 &0 & -a_2\\
a_2 & 0 &-a_0 & a_5 & 0 &-a_3\\
a_3 & 0 & 0 & -a_0 & a_5 & -a_4
\end{pmatrix}
\end{equation}
has maximal rank if and only if $P\in \p^5\setminus L$.
In fact, it is an easy exercise to show that
\begin{equation}
\rk(A)<3 \iff P\in L
\end{equation}
and, if $a_0=a_5=0$
\begin{equation}
\rk(A)<2 \iff P\in C,
\end{equation}
and the thesis follows.

\end{proof}

Now, in order to obtain a congruence $B$, we have to intersect $\Gamma$ with a
hypersurface. Since we are interested in first order congruences, we may
need  that
the congruence $B$ splits in some components.

We have performed the computations in the following examples with the help of
Macaulay2 \cite{M2}.

We note first that $\Gr(1,L)$ is the intersection of $\Gr(1,5)$ with the following
$5$-space:
\begin{equation}\label{eq:klein}
\Pi_L=V(p_{01},p_{02},p_{03},p_{04},p_{05},p_{15},p_{25},p_{35},p_{45})
\end{equation}
and the hyperplane  $H_L:=V(p_{05})$ is the only one which is
tangent to  $\Gr(1,5)$ along $\Gr(1,L)$.

\begin{ex}\label{ex:1}
We first intersect $\Gamma$ with a general hyperplane $H$, in
particular we request that $H \not\supset \Gr(1,L)$. This is case
5.2.1 of \cite{dm}. In this case the congruence $B$ is
irreducible, and the focal locus is $F=X\cup L'$,  where $L'$ is a
scheme whose support is $L$, having $C$ as an embedded component,
and $X$ is a sextic threefold, with Hilbert polynomial
$P_X(t)=t^3+3t^2+2$. $X\cap L$ is a quartic surface with $C$ as
singular locus, and $L'$ is parasitical: therefore the lines of $B$
are the $4$-secants to $X$. The multidegree of $B$ is $(1,3,2)$,
so the hyperplane section of $X$, $S$, is contained in a cubic,
while $X$ is not. The singular locus of $X$ is just the twisted
cubic $C$. The sectional genus is one: hence applying the
quadruple point formula for a curve in $\p^3$ (see \cite{LB}), we
see that the curve section $C_H$ of $X$ has three quadrisecant
lines: two are the lines of the congruence, and the third is
$H\cap L$, \ie a parasitical line.

$X$ is lCM: in fact, its ideal sheaf $\II_X$ has a minimal free
resolution of the form
\begin{equation}\label{seq:lCM}
0\to\OO(-8)\to\OO^6(-7)\xrightarrow{A}\OO^{16}(-6)\to\OO^{22}(-5)
\to\OO^{12}(-4)\to\II_X
\end{equation}
and the points where $X$ is not Cohen-Macaulay
are the ones defined by the first
nonzero Fitting ideal of $A$ (see for example \cite{RS}). It results that the map $A$
has a nonzero kernel, and the $5\times 5$-minors of $A$
define an irrelevant ideal, therefore $X$ is lCM.

As an explicit example, let us take
$H=V(p_{05}+p_{12}+p_{14}+p_{34})$  so the quadric $Q$, component
of the cubic surface $Y$, has equation $a^2+b^2+c^2+d^2-ac$, \ie
the Pfaffian divided by $d$ of the matrix
\begin{equation}
A(d)=\begin{pmatrix} 0 & a & b & c & 0 & d\\
      & 0 & d & 0 & d & 0\\
&  & 0 & 0  & 0 & a\\
&   &  & 0 & d & b\\
&  &  &  & 0 & c\\
&  &  &  &  & 0
\end{pmatrix}.
\end{equation}

Finally, we observe that the lines of $B$ passing through a
general point $P\in X$ form a planar pencil since $B$ is a linear
congruence (see Proposition~4.2 of \cite{dm}).
\end{ex}

\begin{ex}
If  we intersect $\Gamma$ with a hyperplane $H$ such that
$H \supset \Gr(1,L)$ but $H$ is not tangent to $\Gr(1,5)$ along
$\Gr(1,L)$ (\ie $H\neq H_L$), then $B=B'\cup \Gr(1,L)$; the
multidegree of $B'$ is then $(1,3,1)$. The skew-symmetric
matrix associated to $\check\Delta$ (where $\Delta=\Gamma\cap H$
as usual) is of the form
\begin{equation}
A(a_0,\dotsc,a_8)=\begin{pmatrix} 0 & a_0 & a_1 & a_2 & a_3 & a_4\\
      & 0 & 0 & 0 & 0 & a_5\\
&  & 0 & 0  & 0 & a_6\\
&   &  & 0 & 0 & a_7\\
&  &  &  & 0 & a_8\\
&  &  &  &  & 0
\end{pmatrix}.
\end{equation}
The $\p^8$ of these matrices is the embedded tangent space to $\Gr(3,5)$
at the point corresponding to the hyperplane $H_L=V(p_{05})$ of $\p^{14}$,
\ie $\TT_L\Gr(3,5)$.
 In particular, since $\rk(A(a_0,\dotsc,a_8))\le 4$,
the $3$-space $\check\Delta$ is contained in $\check \Gr(1,5)$,
\ie $\check\Delta$
is a $\p^3$ of matrices of rank at most $4$.
Since the intersection of $\Gr(3,5)$ with its tangent space at a point  $P$
is a cone of vertex $P$ over $\p^1\times\p^3$, it has degree $4$ and dimension
$5$, hence, for
general $\Delta$, $\check\Delta\cap\Gr(3,5)$ is given by $4$ points. These
four points correspond to $4$ $\p^3$'s in $\p^5$, that result to be the
components
of the pure focal locus of $B'$, and the lines of $B'$ are the
quadrisecants of these $\p^3$'s.
We recall that a point in $\Gr(3,5)\cap\TT_L\Gr(3,5)$ represents a $\p^3$
intersecting $L$ along a $2$-plane. This implies that the 4 $\p^3$'s are
not in general position, but they intersect two by two in a line of $L$ and
three by three at a
point in $L$.
$L$  turns out to be a parasitical component of the
focal locus of $B'$. This parasitical component has multiplicity two since
through every point of $L$ there passes a $2$-dimensional family of lines of
$B'$.

\begin{rmk} This example was not considered in \cite{dm},
where only the case of a general hyperplane $H$ was studied.
Indeed in that article, in Proposition 5.3, we erroneously claimed
that, if a $3$-space  is contained in $\check \Gr(1,5)$, then the
corresponding focal locus has dimension $>3$. The same error is
contained in \cite{fm},  proof of Theorem 4.3 (as a matter of facts,
Proposition 5.3 of \cite{dm} is a
particular case of Theorem 4.3).
The present case is the only exception
to the quoted Theorem of \cite{fm} (see \cite{fm2}).
\end{rmk}

An explicit example can be taken with $H=V(p_{04}+p_{15})$, which gives the
matrix
\begin{equation}
A(d)=\begin{pmatrix} 0 & a+d & b & c & 0 & 0\\
      & 0 & 0 & 0 & 0 & 0\\
&  & 0 & 0  & 0 & a\\
&   &  & 0 & 0 & b\\
&  &  &  & 0 & c+d\\
&  &  &  &  & 0
\end{pmatrix}.
\end{equation}
The four singular hyperplanes are given by
$H_0=V(p_{25}-p_{45})$, $H_1=V(p_{02}+p_{03}+p_{25}+p_{35})$,
$H_2=V(p_{02}-p_{03}-p_{25}+p_{35})$, and $H_3=V(p_{01}-p_{03})$. They correspond
to the four $\p^3$'s, respectively, $L_0=V(x_5,x_2-x_4)$,
$L_1=V(x_0-x_5,x_2+x_3)$, $L_2=V(x_0+x_5,x_2-x_3)$, and $L_3=V(x_0,x_1-x_3)$.
$\II(B')$ modulo $\II(\Gr(1,5))$ is
\begin{multline}
A=(p_{01}+p_{45}, p_{02}+p_{35}, p_{03}+p_{45}, p_{25}-p_{45},\\
p_{13}p_{23}-p_{12}p_{24}-p_{23}p_{24}+p_{13}p_{34}).
\end{multline}
\end{ex}

\begin{ex}\label{ex:3}
If we intersect $\Gamma$ with
the hyperplane $H_L$, then---set-theoretically---$B=B''\cup \Gr(1,L)$,
but $ \Gr(1,L)$ has to be considered with
multiplicity two; the multidegree of $B''$ is then $(1,3,0)$.
The focal locus of $B''$ is in this case a scheme having support on
$L$, with $C$ as an embedded component; $B''$ is geometrically
given in the following way: consider an isomorphism from the pencil of
hyperplanes containing $L$, $\p^1_L$, to $C$,
$\phi\colon \p_L^1 \to C$. Then
\begin{equation}
B''=\overline{\cup_{H\in\p^1_L}\p^2_{\phi(H)}}
\end{equation}
where $\p^2_{\phi(H)}$ is the star of lines in $H$ with support in
the point $\phi(H)$.

In terms of coordinates, $\II(B'')$ modulo $\II(\Gr(1,5))$ is
\begin{multline}\label{eq:D0}
A_0=(p_{01}+p_{25}, p_{02}+p_{35}, p_{03}+p_{45}, p_{05},\\
p_{15}p_{34}-p_{25}p_{24}+p_{35}p_{23},\\
p_{23}^3-p_{13}p_{23}p_{24}+p_{12}p_{24}^2+p_{13}^2p_{34}-2p_{12}p_{23}p_{34}-
p_{12}p_{14}p_{34}),
\end{multline}
\ie $B''$ is a nonproper intersection of a quadratic and a cubic complex with
$\Gamma$.
\end{ex}

Next theorem collects the results  obtained in the three examples
above.

\begin{thm}\label{thm:es1}
Let $\Gamma$ be the intersection of $\Gr(1,5)$ with the dual of
the plane $\pi_t$ described by equation (\ref{eq:mat}). Let $B$
denote the congruence of lines $\Gamma\cap H$, where $H$ is a
hyperplane of $\p^{14}$. Let $L$ be the $3$-space of equations
$x_0=x_5=0$ and $C\subset L$ be the skew cubic defined by Equation \ref{cubic}.
\begin{enumerate}
\item If $H$ does not contain $\Gr(1,L)$, then $B$ is irreducible of
multidegree $(1,3,2)$. Its focal locus is $F=X\cup L'$, where $L'$
is a non-reduced scheme with $L'_{red}=L$, and $X$ is a lCM
threefold of degree $6$ with $\pi=1$, and with Hilbert polynomial
$P_X(t)=t^3+3t^2+2$. $X\cap L$ is a quartic surface which is
singular along the cubic curve $C$. $L'$ is a parasitical component
of $F$.
\item If $H$ is a general hyperplane containing $\Gr(1,L)$, then
$B=\Gr(1,L)\cup B'$, where $B'$ is a congruence of multidegree
$(1,3,1)$. The focal locus of $B'$ is the union of four $3$-spaces
and a double structure on $L$, which is a parasitical component.
The four $3$-spaces intersect two by two along a line of $L$ and three
by three at a point.
\item If $H=H_L$, the hyperplane which is tangent to  $\Gr(1,5)$ along
$\Gr(1,L)$, then $B=2\Gr(1,L)\cup B''$, where $B''$ is a
congruence of multidegree $(1,3,0)$. The focal locus of $B''$ is a
non-reduced scheme of support $L$, having $C$ as embedded
component.
 \end{enumerate}
\end{thm}

We will consider now the quadratic complexes containing $\Gamma$,
always with the aim of finding first order congruences.
So we take $\Gamma\cap Q$, where $Q$ is a quadratic complex. If
$Q$ is general, then $\Gamma\cap Q$ is an irreducible congruence
of order $2$, hence the focal locus has dimension $4$ by
Theorem~2.1 of \cite{DP6}. In order to get a reducible congruence,
having an irreducible component of order one, we require that $Q$
contains both $B''$ 
and $\Gr(1,L)$. We will
prove now that quadrics of this type do exist.

\begin{thm}\label{thm:quadric} There exists a family of dimension $12$ of quadrics of
$\p^{14}$ containing $\Gr(1,L)\cup B''$. If $Q$ is such a quadric,
then $\Gamma\cap Q=\Gr(1,L)\cup B''\cup B$, where $B$ is a
congruence of multidegree $(1,3,3)$, that is irreducible for
general $Q$. The focal locus of $B$ is a threefold of degre $6$
with $\pi=1$.
\end{thm}
\begin{proof}
The equation $q$ of the quadric $Q$ must be chosen in $A_0$ of
Equation \eqref{eq:D0}, \ie it must be a combination of the
polynomials $p_{01}+p_{25}, p_{02}+p_{35}, p_{03}+p_{45}, p_{05},$
and $p_{15}p_{34}-p_{25}p_{24}+p_{35}p_{23}$.
Let us observe that in this case $Q$ automatically contains the
linear span of $\Gr(1,L)$. Let us consider $Q\cap \check\pi_t$:
its ideal, modulo $\II(\check\pi_t)$, is generated by  $p_{05}$
and $p_{12}p_{04}-p_{23}p_{35}+p_{13}p_{45}$, therefore the family
of congruences we are looking for is the residual of
$B''\cup\Gr(1,L)$ in the intersection
\begin{equation}\label{eq:ideal}
Q(d;a_1,\dotsc,a_6;b_1\dotsc,b_5;c)\cap \Gamma =
V(p_{01}+p_{25},p_{02}+p_{35},p_{03}+p_{45},q)
\end{equation}
where
\begin{multline*}
q=q(d;a_1,\dotsc,a_6;b_1\dotsc,b_5;c):=-d(
p_{15}p_{34}-p_{25}p_{24}+p_{35}p_{23})+\\
+p_{05}(a_1p_{12}+a_2p_{13}+a_3p_{14}+a_4p_{23}+a_5p_{24}
+a_6p_{34}+\\
-(b_1p_{15}+b_2p_{25}+b_3p_{35}+b_4p_{45})+b_5p_{04}-cp_{05}).
\end{multline*}
The multidegree of $B$ is obtained simply by subtracting to the
one of $\Gamma\cap Q$, which is $(2,6,4)$, those of $\Gr(1,L)$ and
of $B''$. The fact that in general $B$ is an irreducible
congruence
 has been checked by a direct computation with Macaulay 2. By
Corollary~\ref{cor:one} it follows that the focal locus has degree
$6$ and sectional genus $1$.
\end{proof}


\begin{rmk} We observe that $Q$ can have rank equal to $8,7,6$ or (at most)
$2$. It has rank $6$ if and only if $a_i=b_j=c=0$, $\forall i,j$
and $d\neq 0$. If $d=0$, the quadric has rank at most $2$. The
general case of rank $8$ is obtained if and only if $d\neq 0$ and
at least one among $a_1,a_2,a_3,b_4,b_5$ is not zero. Finally, if
$d\neq 0$ and $a_1=a_2=a_3=b_4=b_5=0$ (but at least one of the
other terms is different from zero), the quadric has rank $7$.

Clearly, the case of rank $2$ corresponds to reducible quadratic
congruences (which contain the tangent linear congruence of
Theorem~\ref{thm:es1}, case (3)).
\end{rmk}

We will describe in detail the focal locus of the congruence $B$
in the next section.





\section{On the completely exceptional Monge-Amp\`ere
  Equations}\label{sec:3}

For the definitions and generalities about systems of conservation
laws and their correspondence with congruences of lines, we refer
to \cite{AF2}, see also \cite{dm}.
 An important class of strictly
hyperbolic PDE's of conservation laws are the \emph{completely
exceptional Monge-Amp\`ere equations}, which are introduced and
studied in \cite{B}.
 It is shown there that they are \emph{linearly degenerate}, which means that
 the eigenvalues
 of the Jacobian matrix associated to the system
 are constant along the rarefaction curves.
 In \cite{AF2}, it is asserted that  this class is even
a \emph{$T$-system}, \ie that the rarefaction curves are lines,
and it is suggested that this could be an example of a non-linear
$T$-system. We will show that this is the case for systems of the
fourth order; a similar statement can be proven analogously in
general.

The completely exceptional Monge-Amp\`ere systems of the fourth
order are defined as follows (see \cite{AF2},
    Section~5). Introduce the \emph{Henkel (or persymmetric) matrix}
\begin{equation*}
H:=\begin{pmatrix}
\frac{\partial^4 u}{\partial x^4}&\frac{\partial^4 u}{\partial
    x^3\partial t}&\frac{\partial^4 u}{\partial x^2\partial t^2}\\
\frac{\partial^4 u}{\partial x^3\partial t}&\frac{\partial^4
    u}{\partial x^2\partial t^2}&\frac{\partial^4 u}{\partial x\partial
    t^3}\\
\frac{\partial^4 u}{\partial x^2\partial t^2}&\frac{\partial^4
u}{\partial x\partial t^3}&\frac{\partial^4 u}{\partial t^4},
\end{pmatrix}
\end{equation*}
 and consider the PDE of the fourth
order
\begin{equation}\label{eq:1}
d\det(H)+a_1((\frac{\partial^4 u}{\partial
x^4}\frac{\partial^4 u}{\partial x^2\partial
t^2}-(\frac{\partial^4 u}{\partial
    x^3\partial t})^2))+\dotsb+b_1\frac{\partial^4 u}{\partial x^4}+\dotsb +c =0.
\end{equation}
\ie a linear combination of the minors of all orders of $H$ (where
we suppose $d\neq 0$). After the introduction of the new variables
$(u_1,u_2,u_3,u_4,u_5)$ such that
\begin{align*}
u_1&=\frac{\partial^4 u}{\partial x^4},&
u_2&=\frac{\partial^4 u}{\partial x^3\partial t},&
u_3&=\frac{\partial^4 u}{\partial x^2\partial t^2},&
u_4&=\frac{\partial^4 u}{\partial x\partial t^3},&
u_5&=\frac{\partial^4 u}{\partial t^4},
\end{align*}
Equation~\eqref{eq:1} becomes
  \begin{equation}\label{eq:det}
d\det \begin{pmatrix}
u_1&u_2&u_3\\
u_2&u_3&u_4\\
u_3&u_4&u_5\\
\end{pmatrix}+a_1(u_1u_3-u_2^2)+\dotsb+b_1u_1\dotsb+c=0,
\end{equation}
moreover
\begin{align}\label{eq:ui}
(u_1)_t&=(u_2)_x,& (u_2)_t&=(u_3)_x,& (u_3)_t&=(u_4)_x,& (u_4)_t&=(u_5)_x.
\end{align}
This is a system of conservation laws, and the corresponding
congruence $B_{MA}$ in $\p^5$, according to the Agafonov-Ferapontov
construction \cite{AF2} is (in non-homogeneous coordinates
$(y_0,\dotsc,y_4$))
\begin{align*}
y_1&=u_1y_0-u_2\\
y_2&=u_2y_0-u_3\\
y_3&=u_3y_0-u_4\\
y_4&=u_4y_0-u_5
\end{align*}
together with Equation \eqref{eq:det},
or, using projective coordinates $(y_0:\dotsb:y_5)$ in $\p^5$ and
$(u_0:\dotsb:u_5)$ as parameters for the lines in $B_{MA}$, it is given by
\begin{align*}
u_0y_1&=u_1y_0-u_2y_5\\
u_0y_2&=u_2y_0-u_3y_5\\
u_0y_3&=u_3y_0-u_4y_5\\
u_0y_4&=u_4y_0-u_5y_5.\\
\end{align*}

\begin{thm}\label{thm:BMA}
The congruence $B_{MA}$ corresponding to a completely exceptional
Monge-Amp\`ere system coincides with a congruence $B$ of Theorem
\ref{thm:quadric}. In particular $B_{MA}$ has  multidegree
$(1,3,3)$.
\end{thm}

\begin{proof} To obtain the equations of $B_{MA}$ in the Pl\"ucker
embedding, we notice that the line corresponding to the parameters
$(u_0:\dotsb:u_5)$ joins the points $(u_0:\dotsb:u_4:0)$ and
$(0:-u_2:\dotsb:-u_5:u_0)$. Hence we can compute its Pl\"ucker
coordinates taking the $2\times 2$-minors of the matrix
\begin{equation*}
M=\begin{pmatrix}
u_0 &u_1&\dotsb &u_4 &0\\
0&-u_2&\dotsb&-u_5&u_0
\end{pmatrix}.
\end{equation*}
We see immediately that they satisfy the equations in
\eqref{eq:ideal} of $\Gamma\cap Q=\Gr(1,L)\cup B''\cup B$. In
particular, Equation~\eqref{eq:det}, made homogeneous of degree
$4$, becomes the equation of the quadric $Q$. The components
$\Gr(1,L)$ and $B''$ are both contained in the hyperplane
$V(p_{05})$, but the coordinate  $p_{05}$ of a line in $B_{MA}$
has $p_{05}\neq 0$. This implies that $B_{MA}=B$, in particular
the multidegree of $B_{MA}$ is $(1,3,3)$. \end{proof}

\begin{rmk} Note that the degree of $B$ in the Pl\"ucker embedding is
$16$. \end{rmk}



By the above description of the congruences $B$ and $B_{MA}$, we
can find their focal loci in the following way. By definition, the
focal locus is the image of the ramification divisor of the map
$f\colon\Lambda\to \p^5$, where $\Lambda\subset B\times \p^5$ is
the incidence variety and $f$ is the restriction of the
projection. We can work locally, so we can choose local
coordinates $u_1,\dotsc,u_4$ on an open (analytical) subset of $B$
(since by \eqref{eq:det} it  can be written as
$u_5=g(u_1,\dotsc,u_4)$) and $y_0,\dotsc,y_4$ as affine
coordinates on $\p^5$, so $y_0,u_1,\dotsc,u_4$ are local
coordinates on an open subset of $\Lambda$. With this choice, the
Jacobian matrix of $f$ is
\begin{equation*}
Df=
\begin{pmatrix}
\partial f/\partial y_0\\
\partial f/\partial u_1\\
\vdots\\
\partial f/\partial u_4\\
\end{pmatrix}=
\begin{pmatrix}
1 & u_1 & u_2 & u_3 & u_4 \\
0 & y_0 & 0 & 0 & -g_1 \\
0 & -1 & y_0 & 0 & -g_2 \\
0 & 0 & -1 & y_0 & -g_3 \\
0 & 0 & 0 & -1 & y_0-g_4
\end{pmatrix}
\end{equation*}
where $g_i=\partial g/\partial y_i$. Then, the ramification
divisor is just the closure of $V(\det(Df))\cap \Lambda$ and to find the focal
locus $X$ we have to eliminate the variables $u_i$'s. Actually,
the computations were made by observing that
\begin{equation}\label{eq:maca}
\det(Df)=
\det
\begin{pmatrix}
y_0 & 0 & 0 & 0 & h_1\\
-1 & y_0 & 0 & 0 & h_2 \\
0 & -1 & y_0 & 0 & h_3 \\
0 & 0 & -1 & y_0 & h_4 \\
0 & 0 & 0 & -1 & h_5
\end{pmatrix}
\end{equation}
where $h$ is the polynomial defined in \eqref{eq:det} and
$h_i=\partial h /\partial u_i$, \ie the matrix is the matrix of the partial
derivatives of the polynomials defining $\Lambda$ with respect to
$u_1,\dotsc, u_5$. Equation~\eqref{eq:maca} can be shown by some simple but
tedious calculations or by implicit function arguments.

With the help of Macaulay $2$ we get:
\begin{prop} The focal
locus $X$ of $B_{MA}$ has the same invariants as the sextic
threefold of Example~\ref{ex:1}: its Hilbert polynomial is
$P_X(t)=t^3+3t^2+2$ and its singular locus is a twisted cubic. Its
ideal sheaf has the same resolution as in \eqref{seq:lCM}, and by
the same reasons it is lCM.
\end{prop}

\begin{rmk} For particular choices of the quadric in \eqref{eq:ideal} the
threefold $X$ can degenerate: for instance, if $Q$ is a quadric
such that $a_2=\dotsb a_6=b_1=\dotsb=b_5=c=0$, then $X$ is a
scheme with support a cone over a twisted cubic. \end{rmk}

\begin{prop}\label{pr:pl}
The lines of the congruence $B$ passing through a general focal point
$P\in X$ form a planar pencil.
\end{prop}

\begin{proof}
It follows immediately from Lemma~\ref{lem:gp}, observing that,
since  $P$ is general, it is not contained in $L$, and that, since
it is focal, it belongs to infinitely many lines of $B$.
\end{proof}

\begin{cor}
The system of conservation laws defined by \textup{\eqref{eq:det}}
and \textup{\eqref{eq:ui}} is a non linearizable $T$-system, \ie
the corresponding congruence of lines is not contained in a linear
congruence.
\end{cor}

\begin{proof}
It has been shown, for example in \cite{AF2}, Section~2, that a
$T$-system is characterized by the fact that each focus of the
corresponding congruence $B$ is a fundamental point, and moreover
that the lines of $B$, passing through a general focal point, form a
planar pencil. In our case, in view of Proposition \ref{pr:pl}, it
is enough to note that the congruence is not contained in any linear
congruence, because it has multidegree $(1,3,3)$ by Theorem
\ref{thm:BMA}, and  is irreducible.
\end{proof}

We conclude by observing some interesting geometric properties of
$B$. As a subvariety of $\Gr(1,5)$ it is embedded in $\p^{14}$ and
it results to be a $4$-dimensional variety covered by lines, \ie
the planar pencils of Proposition~\ref{pr:pl}.

\begin{thm}
Let $B$ be an irreducible congruence of lines in $\p^5$ of
multidegree $(1,3,3)$ constructed as in Theorem~\ref{thm:quadric}
for a general choice of the quadric $Q$.
\begin{enumerate}
\item Through a general point
of $B$ there pass $4$ lines contained in $B$.
\item $B$ is a smooth Fano fourfold in $\p^{11}$ of index
$2$, of degree $16$ and sectional genus $9$.
\end{enumerate}
\end{thm}

\begin{proof}
The first assertion follows from the fact that on a general line
of $B$ there are $4$ foci (see \cite{DP4}, Proposition~2.1), and
through each of these foci there is a planar pencil of lines of
$B$. As for the second assertion, we have performed the proof with
the help of Macaulay 2 on a general example. We have computed the
equations and the Hilbert polynomial for $B$, obtaining in
particular that its sectional genus is $9$. So the general curve
section of $B$ is a curve in $\p^8$ of degree $16$ and genus $9$
hence is a canonical curve, and $B$ results to be a Fano fourfold
in $\p^{11}$ of index $2$.

For this general choice we have checked that $B$ is smooth.
\end{proof}

It would be desirable to have a more theoretic proof of the preceding theorem.
Nevertheless, we think that this result is interesting because it gives
another interpretation of the Fano fourfold of index two and genus $9$.
This Fano fourfold $B$ belongs to the list of Mukai, see for example
\cite{Mu}, who gives an embedding in $\Gr(2,5)$, the Grassmannian of
planes in $\p^5$. Our result shows that on $B$ there is a rank two vector
bundle giving an embedding in $\Gr(1,5)$. We plan to return to this
argument in a subsequent paper.

\providecommand{\bysame}{\leavevmode\hbox to3em{\hrulefill}\thinspace}


\begin{thebibliography}{Kum66}


\bibitem[AF01]{AF2}
S. I. Agafonov and E. V. Ferapontov,
\emph{Systems of conservation laws of Temple class, equations of
        associativity and linear congruences in $\mathbb{P}^4$},
{Manuscripta Math.} \textbf{106}
({2001}), no.~4, {461--488}.

\bibitem[Boi92]{B}
G. Boillat,  \emph{Sur l'\'equation g\'en\'erale de Monge-Amp\`ere
d'ordre sup\'erieur},  C. R. Acad. Sci. Paris SŽr.~I  Math.  \textbf{315}
(1992),  no. 11, 1211--1214.




\bibitem[Cha89]{Cha}
M.C. Chang,
\emph{Classification of Buchsbaum subvarieties of
codimension $2$ in projective space},
J. Reine Angew. Math. \textbf{401} (1989), 101--112.


\bibitem[CC93]{CC}
L. Chiantini and C. Ciliberto,
\emph{A few remarks on the lifting problem},
Journ\'ees de G\'eom\'etrie Alg\'ebrique d'Orsay (Orsay, 1992),
Ast\'erisque No. 218 (1993), 95--109.


\bibitem[DP04]{DP6}
 P. De Poi,  \emph{Congruences of lines with one-dimensional focal locus},
{Port. Math. (N.S.)}  \textbf{61}  ({2004}), no.~3, 329--338.

\bibitem[DP05]{DP4}
\bysame, \emph{On first order congruences of lines in $\mathbb{P}^4$ with
 irreducible fundamental surface},
{Math. Nachr.} \textbf{278} (2005), no.~4, 363--378.


\bibitem[DM05]{dm}
P. De Poi and E. Mezzetti, \emph{Linear congruences and systems of
   conservation laws}, Projective Varieties with Unexpected Properties
    (A Volume in Memory of Giuseppe Veronese.
Proceedings of the international conference
``Varieties with Unexpected Properties'',
Siena, Italy, June 8--13, 2004), de Gruyter, Berlin, 2005, pp.~209--230.


\bibitem[FM02]{fm}
M.L. Fania and E. Mezzetti,
\emph{On the Hilbert scheme of Palatini threefolds},
{Adv. Geom.} \textbf{2}
(2002), no.~4, 371--389.

\bibitem[FM08]{fm2}
\bysame,
\emph{Erratum to \lq\lq\thinspace On the Hilbert scheme of Palatini threefolds''},
{Adv. Geom.} \textbf{8}
(2008), 153--154 (to appear).

\bibitem[M2]{M2}
 {D.R. Grayson and M.E.  Stillman},
\emph{Macaulay 2, a software system for research
                   in algebraic geometry},
 {Available at \url{http://www.math.uiuc.edu/Macaulay2/}}. 


\bibitem[GLP83]{GLP}
L. Gruson, R. Lazarsfeld and C. Peskine,
\emph{On a theorem of Castelnuovo, and the equations defining space curves},
 Invent. Math.  \textbf{72}  (1983),  no. 3, 491--506.

\bibitem[GP82]{gp}
L. Gruson and C. Peskine, \emph{Section plane d'une courbe gauche:
postulation}, Enumerative geometry and classical algebraic
geometry (Nice, 1981); Progr. Math., Vol. 24, Birkh\"auser, Boston
Mass., 1982, 33--35.




\bibitem[Har77]{H}
R.~Hartshorne, \emph{Algebraic Geometry}, Graduate Texts in Mathematics,
Vol. 52, Springer Verlag, 1977.


\bibitem[Kwa99]{KW}
S.J. Kwak,
\emph{Castelnuovo-Mumford regularity bound for smooth threefolds in $\p^5$
and extremal examples},
J. Reine Angew. Math. \textbf{509} (1999), 21--34.

\bibitem[Koe92]{Ko}
L. Koelblen,
\emph{Surfaces de $P\sb 4$ trac\'ees sur une hypersurface cubique},
J. Reine Angew. Math. \textbf{433} (1992), 113--141.


\bibitem[Lau78]{lau}
O. A. Laudal,
\emph{A generalized trisecant lemma},
Algebraic geometry (Proc. Sympos., Univ. Troms\o, Troms\o, 1977);
  Lecture Notes in Math., Vol. 687,
Springer Verlag, Berlin, 1978, 112--149.


\bibitem[LeB82]{LB}
P. Le Barz, \emph{Formules multis\'ecantes pour les courbes
gauches quelconques}, Enumerative geometry and classical algebraic
geometry (Nice, 1981), pp. 165--197, Progr. Math., 24, BirkhŠuser,
Boston, Mass., 1982.

\bibitem[MM05]{MM}
      {L. Manivel and E. Mezzetti},
\emph{On linear spaces of
skew-symmetric matrices of constant rank},
{Manuscripta Math. \textbf{117} (2005), n.3, 319--331.}

\bibitem[Mez92]{Mez2}
E. Mezzetti,
\emph{The border cases of the lifting theorem for surfaces in $\p^4$},
J. Reine Angew. Math. \textbf{ 433} (1992), 101--111.


\bibitem[Mez94]{Mez}
   \bysame,
  \emph{Differential-geometric methods for the lifting problem
and linear systems
    on plane curves},
  J. Algebraic Geom. \textbf{3} (1994) no. 3, 375--398.


\bibitem[MR90]{MR}
E.~Mezzetti and I.~Raspanti,
\emph{A Laudal-type theorem for surfaces in $\p^4$},
Commutative algebra and algebraic geometry (Turin, 1990),
Rend. Sem. Mat. Univ. Politec. Torino \textbf{48} (1990), no. ~4, 529--537

\bibitem[Mu88]{Mu}
S.~Mukai,
\emph{Curves, $K3$ surfaces and Fano $3$-folds of genus $\leq 10$},
in Algebraic geometry and commutative algebra, Vol.~I, 357--377,
Kinokuniya, Tokyo, 1988.



\bibitem[Ott92]{O}
G.~Ottaviani, \emph{On $3$-folds in $\mathbb{P}^5$ which are scrolls}, Ann.
  Scuola Norm. Sup. Pisa, Cl. Sci. \textbf{19} (1992), 451--471.

\bibitem[Rog01]{Rug1}
M. Roggero,
\emph{Lifting problem for codimension two subvarieties in $\p^{n+2}$:
   border cases},
Geometric and combinatorial aspects of commutative algebra (Messina,
   1999);
Lecture Notes in Pure and Appl. Math., Vol. 217,
Dekker, New York, 2001, 309--326.



\bibitem[Rog03]{Rug}
\bysame, \emph{Laudal-type theorems in $\p^N$}, Indag. Math.
\textbf{14} (2003), no. 2, 249--262.

\bibitem[Ro03a]{rug}
\bysame, \emph{Generalizations of ``Laudal Trisecant Lemma'' to
codimension $2$ subvarieties in $\p^n$}, Quaderni del Dipartimento
di Matematica di Torino, 23/2003.

\bibitem[RS91]{RS}
F. Rossi and W. Spangher, \emph{Computation of the openness of
some loci of modules},  Applied algebra, algebraic algorithms and
error-correcting codes (New Orleans, LA, 1991),  390--402,
Lecture Notes in Comput. Sci., 539, Springer, Berlin, 1991.


\bibitem[Sch90]{Sch}
M. Schneider,
\emph{Vector bundles and low-codimensional submanifolds of
projective space: a   problem list},
  Topics in algebra, Part 2 (Warsaw, 1988),
Banach Center Publ.26, 209--222, PWN Warsaw, 1990.


\bibitem[Sev01]{Sev}
F. Severi,
  \emph{Intorno ai punti doppi impropr\^\i{} di una superficie generale dello
     spazio a quattro dimensioni, e a' suoi punti tripli apparenti}
  Palermo Rend. \textbf{15}, (1901), 33--51; also in \emph{Opere matematiche: memorie e note}. Vol. I: 1900--1908, Accademia Nazionale dei Lincei, Roma, 1971, 14--30.




\bibitem[Str87]{str}
{R. Strano},
\emph{On the hyperplane sections of curves},
Proceedings of the Geometry Conference (Milan and Gargnano, 1987),
Rend. Sem. Mat. Fis. Milano
\textbf{57}  (1989), 125--134.



\bibitem[Zak97]{Z}
{F.~L. Zak}, \emph{Congruences of lines and extrinsic geometry of projective
varieties of small codimension}, 11--14 June 1997, talk at the Meeting on
Syzygies, Catania. URL address:
{\url{http://euclid.mathematik.uni-kl.de/~pragma/1997}}.
\end{thebibliography}
\end{document}